\documentclass[11pt]{article}
\usepackage{amsmath,amsfonts,amssymb, amsthm,enumerate,graphicx}
\usepackage{tikz,authblk}
\usepackage{subfig}
\usepackage{url}

\newtheorem{theorem} {{\textsf{Theorem}}}
\newtheorem{proposition}[theorem]{{\textsf{Proposition}}}
\newtheorem{corollary}[theorem]{{\textsf{Corollary}}}

\newtheorem{definition}[theorem]{{\textsf{Definition}}}
\newtheorem{remark}[theorem]{{\textsf{Remark}}}

\newtheorem{lemma}[theorem]{{\textsf{Lemma}}}

\begin{document}

\title{On regular genus and G-degree of PL 4-manifolds with boundary}
\author{Biplab Basak$^1$ and Manisha Binjola}

\date{}

\maketitle

\vspace{-15mm}
\begin{center}

\noindent {\small Department of Mathematics, Indian Institute of Technology Delhi, New Delhi 110016, India.$^2$}


\footnotetext[1]{{\em Corresponding author.}}

\footnotetext[2]{{\em E-mail addresses:} \url{biplab@iitd.ac.in} (B.
Basak), \url{binjolamanisha@gmail.com} (M. Binjola).}

\medskip

\date{December 03, 2022}
\end{center}

\hrule

\begin{abstract}
In this article, we introduce two new PL-invariants: weighted regular genus and weighted G-degree for manifolds with boundary. We first prove two inequalities involving some PL-invariants which state that for any PL-manifold $M$ with non spherical boundary components, the regular genus $\mathcal{G}(M)$ of $M$ is at least the weighted regular genus $\tilde{G}(M)$ of $M$ which is again at least the generalized regular genus $\bar{G}(M)$ of $M$. Another inequality states that the weighted G-degree $\tilde{D}_G (M)$ of $M$ is always greater than or equal to the G-degree $D_G (M)$ of $M$. Let $M$ be any compact connected PL $4$-manifold with $h$ number of non spherical boundary components. Then we compute the following:
$$\tilde{G} (M) \geq 2 \chi(M)+3m+2h-4+2 \hat{m} \mbox{  and } \tilde{D}_G (M) \geq 12(2 \chi(M)+3m+2h-4+2 \hat{m}),$$ where $m$ and $\hat{m}$ are the ranks of the fundamental groups of $M$ and the corresponding singular manifold $\widehat{M}$ (obtained by coning off the boundary components of $M$) respectively. As a
 consequence we prove that the regular genus $\mathcal{G}(M)$ satisfies the following inequality:
 $$\mathcal{G} (M) \geq 2 \chi(M)+3m+2h-4+2 \hat{m},$$
 which improves the previous known lower bounds for the regular genus $\mathcal{G}(M)$ of $M$.
Then we define two classes of gems for PL $4$-manifold $M$ with boundary: one consists of semi-simple gems and the other consists  of weak semi-simple gems, and prove that the lower bounds for the weighted G-degree and weighted regular genus are attained in these two classes respectively. 

\end{abstract}
 
%


\smallskip


\medskip


\section{Introduction}

A crystallization $(\Gamma,\gamma)$ of a PL $d$-manifold is a certain type of edge colored graph which represents the manifold (cf. Subsection \ref{crystal} for details). The crystallization theory was introduced by Pezzana who gave the existence of a crystallization  for every closed connected PL $d$-manifold (see \cite{pe74}). Later,  the existence of crystallization was shown for every compact PL $d$-manifold with boundary (see \cite{cg80, ga83}). 
The term regular genus for a closed connected PL $d$-manifold which extends the notion of genus in dimension 2, has been introduced in \cite{ga81}, which is related to the existence of regular embeddings of gems of the manifold into surfaces (cf. Subsection \ref{sec:genus} for details). Later, in \cite{ga87}, the concept of regular genus has been extended for compact PL $d$-manifolds with boundary, for $d\geq 2$. The regular genus of a closed connected orientable (resp., a non-orientable) surface equals the genus (resp., half of the genus) of the surface. Some beautiful results on  regular genus of mapping tori can be found in \cite{bb19}. In \cite{bc15}, the authors gave lower bound for regular genus of closed $4$-manifolds. In \cite{cp90}, two lower bounds for the regular genus of connected compact PL $d$-manifolds with boundary have been computed. In \cite{bm20}, we gave several lower bounds for compact $4$-manifolds with boundary. In \cite{ccg20}, the authors provided a lower bound for both regular genus and G-degree for compact $4$-manifolds with empty or connected boundary. In this article, we further improve the previous lower bounds for regular genus of compact  PL $4$-manifolds with boundary which is an easy consequence of the main theorem (cf. Theorem \ref{thm:weightedgenus}).

In \cite{cc20}, the authors have introduced the concept of generalized regular genus for compact PL $d$-manifolds with empty or non-spherical boundary components, which coincides with the regular genus for closed PL $d$-manifolds. A lower bound for the generalized regular genus of  compact $4$-manifolds with  at most one boundary component has been computed in \cite{cc20,ccg20}, which also gives a lower bound for the regular genus of  compact $4$-manifolds with  at most one non-spherical boundary component. In this article, we introduce another new invariant weighted regular genus for compact $4$-manifolds with  several non-spherical boundary components, which coincides with the generalized regular genus  when the manifold is a compact $4$-manifold with at most one boundary component. First, we note that for any compact PL $d$-manifold $M$ with boundary, the weighted regular genus of $M$ lies between the regular genus and generalized regular genus of $M$ such that the regular genus being the largest of these three. If $M$ is closed, then all the three PL-invariants coincide. The main contribution of this article is that we give a lower bound for the weighted regular genus of compact PL $4$-manifolds with  several non-spherical boundary components. As a consequence, we give a lower bound for the regular genus  of compact $4$-manifolds with  several boundary components, which improves the previous known estimates  for the regular genus.

The notion of G-degree of closed connected PL $d$-manifold was first introduced in \cite{ccdg18}, which is a key concept in the approach to Quantum Gravity via tensor models. Later, it has been extended for singular manifolds in  \cite{ccg17}, and hence for compact PL $4$-manifolds with  empty or non-spherical boundary as there is a one-to-one correspondence between singular manifolds and compact $4$-manifolds with  empty or non-spherical boundary.  Several combinatorial properties of G-degree  have been studied recently in \cite{cg19,cc19,cc20,ccg20} for  compact $4$-manifolds with  at most one boundary component. In this article, we have introduced a new invariant  weighted G-degree for  compact $4$-manifolds with  several non-spherical boundary components, which coincide with the $G$-degree when the manifold is a compact $4$-manifold with  at most one boundary component. All the PL-invariants mentioned above are non-negative. Also, for any compact PL $d$-manifold $M$ with boundary, weighted G-degree  of the manifold is at least the G-degree of that manifold. In this article, we also give a lower bound for the  weighted G-degree for compact $4$-manifolds with non-spherical boundary components.
 
 In \cite{ccg20}, the class of  semi-simple and weak semi-simple  gems have been introduced for compact $4$-manifolds with empty or connected boundary. In this paper, we extend the classes to compact $4$-manifolds with several boundary components. Then we prove that a compact PL $4$-manifold $M$ with boundary attains the lower bound for the weighted G-degree (resp., weighted regular genus) if and only if $M$ admits a  semi-simple (resp., weak semi-simple) gem.

\section{Preliminaries}
In this section, we shall give a brief overview of crystallization theory. It provides a combinatorial tool for representing piecewise-linear (PL) manifolds of arbitrary dimension via colored graphs and is used to study geometrical and topological properties of manifolds.

\subsection{Crystallization} \label{crystal}

A multigraph is a graph where multiple edges are allowed but loops are restricted. For a multigraph $\Gamma= (V(\Gamma),E(\Gamma))$, a surjective map $\gamma : E(\Gamma) \to \Delta_d:=\{0,1, \dots , d\}$ is called a proper edge-coloring if $\gamma(e) \ne \gamma(f)$ for any adjacent edges $e,f$. The elements of the set $\Delta_d$ are called the {\it colors} of $\Gamma$. A graph $(\Gamma,\gamma)$ is called {\it $(d+1)$-regular} if degree of each vertex is $d+1$ and is said to be {\it $(d+1)$-regular with respect to a color $c$} if after removing all the edges of color $c$ from $\Gamma$, the resulting graph is $d$-regular. We refer to \cite{bm08} for standard terminology on graphs. All spaces and maps will be considered in PL-category.

\smallskip

A  regular {\it $(d+1)$-colored graph} is a pair $(\Gamma,\gamma)$, where $\Gamma$ is $(d+1)$-regular and $\gamma$ is a proper edge-coloring.  A {\it $(d+1)$-colored graph with boundary} is a pair $(\Gamma,\gamma)$, where $\Gamma$ is not a $(d+1)$-regular graph but a $(d+1)$-regular with respect to a color $c$ and $\gamma$ is a proper edge-coloring. If coloration is understood, one can use $\Gamma$ for $(d+1)$-colored graphs instead of $(\Gamma,\gamma)$. For each $B \subseteq \Delta_d$ with $h$ elements, the graph $\Gamma_B =(V(\Gamma), \gamma^{-1}(B))$ is an $h$-colored graph with edge-coloring $\gamma|_{\gamma^{-1}(B)}$. For a color set $\{j_1,j_2,\dots,j_k\} \subset \Delta_d$, $g(\Gamma_{\{j_1,j_2, \dots, j_k\}})$ or $g_{j_1j_2 \dots j_k}$ denotes the number of connected components of the graph $\Gamma_{\{j_1, j_2, \dots, j_k\}}$. Let $\dot{g}_{j_1j_2 \dots j_k}$ denote the number of regular components of $\Gamma_{\{j_1, j_2, \dots, j_k\}}$. A graph $(\Gamma,\gamma)$ is called {\it contracted} if subgraph $\Gamma_{\hat{c}}:=\Gamma_{\Delta_d\setminus c}$ is connected for all $c$.

\smallskip
 
Let $\mathbb{G}_d$ denote the set of graphs $(\Gamma,\gamma)$ which are $(d+1)$-regular with respect to the fixed color $d$. Thus  $\mathbb{G}_d$ contains all the regular $(d+1)$-colored graphs as well as all $(d+1)$-colored graphs with boundary. We sometimes call a regular $(d+1)$-colored graph or a $(d+1)$-colored graphs with boundary simply by a $(d+1)$-colored graph if it is understood that the graph is regular or not. If $(\Gamma,\gamma)\in \mathbb{G}_d$ then the vertex with degree $d+1$ is called an internal vertex and the vertex with degree $d$ is called a boundary vertex. Let $2\overline{p}$ and $2\dot{p}$ denote the number of boundary vertices and internal vertices respectively. Then the number of vertices is denoted by $2p$, i.e., $p=\overline{p}+\dot{p}$.
 For each graph $(\Gamma,\gamma) \in \mathbb{G}_d$,  we define its  boundary graph $(\partial \Gamma,\partial \gamma)$ as follows:

\begin{itemize}
\item{} there is a bijection between $V(\partial \Gamma)$ and the set of boundary vertices of $\Gamma$;

\item{} $u_1,u_2 \in V(\partial \Gamma)$ are joined in $\partial \Gamma$ by an edge of color $j$ if and only if $u_1$ and $u_2$ are joined in $\Gamma$  by a path formed by $j$ and $d$ colored edges alternatively.
\end{itemize} 

\smallskip

Note that, if $(\Gamma,\gamma)$ is $(d+1)$-regular then $(\Gamma,\gamma)\in \mathbb{G}_d$ and $\partial \Gamma = \emptyset$. For each  $(\Gamma,\gamma) \in \mathbb{G}_d$, a corresponding $d$-dimensional simplicial cell-complex ${\mathcal K}(\Gamma)$ is determined as follows:

\begin{itemize}
\item{} for each vertex $u\in V(\Gamma)$, take a $d$-simplex $\sigma(u)$ and label its vertices by $\Delta_d$;

\item{} corresponding to each edge of color $j$ between $u,v\in V(\Gamma)$, identify the ($d-1$)-faces of $\sigma(u)$ and $\sigma(v)$ opposite to $j$-labeled vertices such that the vertices with same label coincide.
\end{itemize}

 The geometric carrier $|\mathcal{K}(\Gamma)|$ is a $d$-pseudomanifold and $(\Gamma,\gamma)$ is said to be a gem (graph encoded manifold) of any $d$-pseudomanifold homeomorphic to $|\mathcal{K}(\Gamma)|$ or simply is said to represent the $d$-pseudomanifold.  We refer to \cite{bj84} for CW-complexes and related notions. From the construction it is easy to see that, for $\mathcal{B} \subset \Delta_d$ of cardinality $h+1$, ${\mathcal K}(\Gamma)$ has as many $h$-simplices with vertices labeled by $\mathcal{B}$ as the number of connected components of $\Gamma_{\Delta_d \setminus \mathcal{B}}$ (cf. \cite{fgg86}).
\begin{definition}
A closed connected PL $d$-manifold is a compact $d$-dimensional polyhedron which has a simplicial triangulation such that the link of each vertex is $(d-1)$-sphere. 

A connected compact PL $d$-manifold with non-empty boundary is a compact $d$-dimensional polyhedron which has a simplicial triangulation where the link of at least one vertex is a  $(d-1)$-ball and the link of each of the other vertices is either a $(d-1)$-sphere or a $(d-1)$-ball.

A singular PL $d$-manifold is a compact $d$-dimensional polyhedron which has a simplicial triangulation where the links of vertices are closed connected $(d-1)$ manifolds while, for each $h\geq 1$, the link of any $h$-simplex is a PL $(d-h-1)$-sphere. A vertex whose link is not a sphere is called a singular vertex. Clearly, a closed (PL) $d$-manifold is a singular (PL) $d$-manifold with no singular vertices.
\end{definition}

From the correspondence between $(d+1)$-regular colored graphs and $d$-pseudomanifolds, it is easy to visualise that:

\begin{enumerate}[$(1)$] 
\item $|\mathcal{K}(\Gamma)|$ is a closed connected PL $d$-manifold if and only if for each $c\in \Delta_d$, $\Gamma_{\hat{c}}$ represents $\mathbb{S}^{d-1}$.

\item $|\mathcal{K}(\Gamma)|$ is a connected compact PL $d$-manifold with boundary if and only if for each $c\in \Delta_d$, $\Gamma_{\hat{c}}$ represents either $\mathbb{S}^{d-1}$ or $\mathbb{B}^{d-1}$, and there exists a vertex $c\in \Delta_d$ such that $\Gamma_{\hat{c}}$ represents $\mathbb{B}^{d-1}$.

\item $|\mathcal{K}(\Gamma)|$ is a singular (PL) $d$-manifold if and only if for each $c\in \Delta_d$, $\Gamma_{\hat{c}}$ represents a closed connected PL $(d-1)$-manifold.
\end{enumerate}

If $(d+1)$-colored graph  $(\Gamma,\gamma)$ which is a gem of a closed (PL) $d$-manifold $M$, is contracted then it is called a {\it crystallization} of $M$. In this case, the number of vertices of ${\mathcal K}(\Gamma)$ is exactly $d+1$ (which is the minimal). It is not hard to see that, for any $(d+1)$-colored graph (with boundary) $(\Gamma,\gamma)$, $|{\mathcal K}(\Gamma)|$ is orientable if and only if $\Gamma$ is a bipartite graph. If $\Gamma$ represents a PL $d$-manifold with boundary then we can define its boundary graph $(\partial \Gamma,\partial \gamma)$, and each component of the boundary-graph $(\partial \Gamma,\partial \gamma)$ represents a component of $\partial M$.

On the other hand, if  $(\Gamma,\gamma)\in \mathbb{G}_d$ represents a $d$-manifold with non-connected boundary then $\Gamma$ can not be contracted. Let $(\Gamma,\gamma)\in \mathbb{G}_d$ represent a $d$-manifold with boundary. Then, it is easy to see that $\Gamma_{\hat d}$ is connected and each component of $\partial \Gamma$ is contracted if and only if $\mathcal{K}(\Gamma)$ has exactly $dh+1$ vertices (which is the minimal), where $h$ is the number of components of $\partial M$. 

Let $M$ be a compact $d$-manifold with $h$ boundary components. A $(d+1)$-colored gem $(\Gamma,\gamma) \in \mathbb{G}_d$ of $M$ is said to be a {\it crystallization} of $M$ if $g(\Gamma_ {\hat{d}})=1$ and $g(\Gamma_ {\hat{c}})=h$, for each $c \in \Delta_{d-1}$.

The initial point of the crystallization theory is the Pezzana's existence theorem (cf. \cite{pe74}) which gives the existence of a crystallization for every closed connected PL $d$-manifold. Later, it has been extended to the boundary case (cf \cite{cg80, ga83}). Further, the existence of crystallizations/gems has been extended for singular (PL) $d$-manifolds (cf. \cite{ccg17}). It is known that a PL $d$-manifold with boundary can always be represented by a $(d+1)$-colored graph $(\Gamma,\gamma)$ which is regular with respect to a fixed color $k$, for some $k\in \Delta_{d}$. Without loss of generality, we can assume that $k=d$, i.e., $(\Gamma,\gamma) \in \mathbb{G}_d$. 

A $1$-dipole of color $j \in \Delta_d$ of a $(d+1)$-colored graph (possibly with boundary) $(\Gamma,\gamma) \in \mathbb{G}_d$ is a subgraph $\theta$ of $\Gamma$ consisting of two vertices $x,y$ joined by color $j$ such that 
$\Gamma_{\hat{j}} (x) \neq \Gamma_ {\hat{j}} (y)$, where $\Gamma_ {\hat{j}} (u)$ denotes the component of $\Gamma_ {\hat{j}}$ containing $u$. The cancellation of a $1$-dipole from $\Gamma$ consists of two steps: first deleting $\theta$ from $\Gamma$ and second welding the same colored hanging edges.(see \cite{bcg01} for more details)

\subsection{Regular Genus of PL $d$-manifolds}\label{sec:genus}

 Let $(\Gamma,\gamma)\in \mathbb{G}_d$ be a bipartite (resp. non bipartite) $(d+1)$-colored graph (with boundary) which represents a connected compact PL $d$-manifold $M$ with boundary (possibly $\partial M=\phi$). Add a new vertex $u^\prime$ for each boundary vertex $u$ and join these two vertices by a $d$-colored edge and we get a new graph $(\Gamma^\prime,\gamma^\prime)$. Then for each cyclic permutation $\varepsilon=(\varepsilon_0,\dots,\varepsilon_d)$ of $\Delta_d$, there exists a regular imbedding of $\Gamma^\prime$ into an orientable (resp. non orientable) surface $F$ such that the intersection of vertices of $\Gamma^\prime$ and $\partial F$ is the set of new added vertices and the regions are bounded either by a cycle (internal region) or by a walk (boundary region) of $\Gamma^ \prime$ with $\varepsilon_j,\varepsilon_{j+1}$(j mod d+1) colored edges alternatively.
 
 Using Gross `voltage theory' (resp. Stahl  `embedding schemes') in  (\cite{gro74}) (resp., \cite{st78}), it is easy to show that for every cyclic permutation $\varepsilon$ of $\Delta_d$, a regular imbedding of bipartite graph (resp. non bipartite) $\Gamma$ into an orientable (resp. non orientable) surface  $F_ \varepsilon$ exists. Moreover, genus (resp. half of genus) $\rho_ \varepsilon$ of $F_ \varepsilon$ satisfies
 $$2-2\rho_ \varepsilon(\Gamma)=\sum_{i \in \mathbb{Z}_{d+1}}\dot{g} _{\varepsilon_i\varepsilon_{i+1}} + (1-d) \ \dot{p} + (2-d) \ \overline{p} +\partial g_{\varepsilon_0\varepsilon_{d-1}}$$
 where $2\dot{p}$ and $2\bar{p}$  denote the number of internal vertices and boundary vertices in $\Gamma$, and $\partial g_{ij}$ denote the number of $\{i,j\}$-colored cycles in $\partial \Gamma$. For more details we refer \cite{bcg01, ga87}.
  
The regular genus $\rho(\Gamma)$ of $(\Gamma,\gamma)$ is defined as
$$\rho(\Gamma)= \min \{\rho_{\varepsilon}(\Gamma) \ | \  \varepsilon =(\varepsilon_0,\varepsilon_1,\dots,\varepsilon_d )\ \mbox{ is a cyclic permutation of } \ \Delta_d\}.$$

In dimension two, it is easy to see that if $(\Gamma,\gamma)$ represents a surface $F$, then the corresponding $(\Gamma',\gamma')$ regularly imbeds into $F$ itself. Hence, for each surface $F$,

\begin{align*}
\mathcal G(F)=
\left\{ \begin{array}{lcl}
genus(F)  & \mbox{if} & F \mbox{ is orientable}, \\
\frac{1}{2} \times genus(F) & \mbox{if} & F \mbox{ is non-orientable}.
\end{array}\right. \nonumber
\end{align*}

Further on similar steps, if $(\Gamma,\gamma)\in \mathbb{G}_d$ is a regular bipartite (resp. non bipartite) $(d+1)$-colored graph of order $2p$ which represents a singular $d$-manifold, then the genus (resp. half of genus) $\rho_\varepsilon$ of surface $F_ \varepsilon$ satisfies
$$2-2\rho_ \varepsilon(\Gamma)=\sum_{i \in \mathbb{Z}_{d+1}}g_{\varepsilon_i\varepsilon_{i+1}} + (1-d)p.$$ 
The regular genus $\rho(\Gamma)$ of $(\Gamma,\gamma)$ is defined as
$$\rho(\Gamma)= \min \{\rho_{\varepsilon}(\Gamma) \ | \  \varepsilon =(\varepsilon_0,\varepsilon_1,\dots,\varepsilon_d )\ \mbox{ is a cyclic permutation of } \ \Delta_d\}.$$

 \begin{definition}
 Let $M$ be a connected compact PL $d$-manifold with boundary (possibly empty). Then, the regular genus of $M$ is defined as 
 $$\mathcal G(M) = \min \{\rho(\Gamma) \ | \  (\Gamma,\gamma)\in \mathbb{G}_d \mbox{ represents } M\}.$$
 \end{definition}

\section{Main results}
\begin{proposition}\label{prop:crys}
 Any compact orientable (resp., non-orientable) (PL) $d$-manifold $M$ with boundary (possibly empty)  admits a bipartite (resp., non-bipartite) $(d+1)$-colored graph with boundary (possibly empty) representing it. Moreover, $M$ admits a crystallization.
  \end{proposition}

\begin{remark} [\cite{ccg20}] {\rm There is a bijection between the class of connected singular (PL) $d$-manifolds and the class of connected closed (PL) $d$-manifolds union with the class of connected compact (PL) $d$-manifolds with non-spherical boundary components. For, if $M$ is a singular $d$-manifold then removing small open neighborhood of each of its singular vertices (if possible), a compact $d$-manifold $\check{M}$ (with non spherical boundary components) is obtained. It is obvious that $M=\check{M}$ if and only if $M$ is a closed $d$-manifold.
 
 Conversely, if $M$ is a compact $d$-manifold with non spherical boundary components then a singular $d$-manifold
 $\widehat{M}$ is obtained by coning off each component of $\partial M$. If $M$ is a closed $d$-manifold then $M= \widehat{M}$.}
 \end{remark}

From now onwards, we mean `a connected compact PL $d$-manifold with non spherical boundary components' by the term `a manifold with boundary'. The boundary component can be empty as well.

 \begin{definition}
 A regular $(d+1)$-colored graph $(\Gamma,\gamma)$ is said to be a regular gem of a compact $d$-manifold $M$ with boundary if $\Gamma$ represents the associated singular $d$-manifold $\widehat{M}$.
 \end{definition}
 
For a colored graph $(\Gamma,\gamma)$, a color $c \in \Delta_d$ is said to be {\em singular}, if the vertices labeled by color $c$ in $\mathcal{K}(\Gamma)$ are singular. Let $M$ be a connected compact $d$-manifold with (non-empty) boundary. Then by Proposition \ref{prop:crys},  there exists a $(d+1)$-colored graph $(\Gamma,\gamma) \in \mathbb{G}_d$ representing $M$. For each component of $\partial \Gamma$, choose $c \in \Delta_ {d-1}$. Join an edge of color $d$ in $\Gamma$ between the two boundary vertices which lie on $\{c,d\}$-colored path in $\Gamma$. The resulting $(d+1)$-regular colored graph represents the singular $d$-manifold $\widehat{M}$ which is obtained by coning off each boundary component of $M$. The colors $c$ which were chosen for different components of $\partial \Gamma$ are  singular. If we choose a fixed color $c \in \Delta_ {d-1}$ for each boundary component then the resulting $(d+1)$-regular colored graph has exactly one color $c$ as a singular color. Replace $c$ and $d$ with each other. Then, we get a $(d+1)$-colored graph $(\tilde{\Gamma},\tilde{\gamma})$ with $d$ as a singular color and the colors in $\Delta_ {d-1}$ are non-singular.
  
 Let $\tilde{\mathbb{G}}_d$ denote the set of all $(d+1)$-regular colored graphs $(\Gamma,\gamma)$ such that vertices in $\mathcal{K}(\Gamma)$ labeled by color $c \in \Delta_{d-1}$ are non-singular and the vertices labeled by color $d$ are singular. Thus, from Proposition \ref{prop:crys}, we have the following result.
  
\begin{corollary}
 Any connected compact orientable (resp., non-orientable) PL $d$-manifold $M$ with boundary admits a regular bipartite (resp., non-bipartite) $(d+1)$-colored gem of $M$. Moreover, there is a regular $(d+1)$-colored gem of $M$ which lies in $\tilde{\mathbb{G}}_d$.
\end{corollary}

With this result, we have two representaions for a connected compact PL $d$-manifold $M$ with non spherical boundary components: one is a non-regular $(d+1)$-colored graph in $\mathbb{G}_d$ and the other is a regular $(d+1)$-colored graph (may not lie in  $\tilde{\mathbb{G}}_d$). Moreover, we  have another regular $(d+1)$-colored gem of $M$, which lies in $\tilde{\mathbb{G}}_d$.

\begin{definition}
Let $(\Gamma,\gamma)$ be a regular $(d+1)$-colored graph. If $\{\varepsilon^{(j)}| j=1,\dots,\frac{d!}{2} \}$ is the set of all permutations of $\Delta_d$ then the Gurau degree (or in short G-degree) $\omega_G (\Gamma)$ of $\Gamma$ is defined as 
$$\omega_G (\Gamma)=\sum_{j=1}^{\frac{d!}{2}} \rho_{\varepsilon^{(j)}}(\Gamma).$$
\end{definition} 
 Now, we define some PL-invariants combinatorially.

\begin{definition}\label{maindef}
Let $M$ be a compact connected PL $d$-manifold with boundary $(d \geq 2)$. Then the generalized regular genus $\bar{\mathcal G}(M)$ of $M$ is defined as 
$$\bar{\mathcal G}(M) = \min \{\rho(\Gamma) \ | \  (\Gamma,\gamma) \mbox{ is a regular gem of } M\};$$
the weighted regular genus $\tilde{G}(M)$ is defined as 
$$\tilde{\mathcal G}(M) = \min \{\rho(\Gamma) \ | \  (\Gamma,\gamma) \in \tilde{\mathbb{G}}_d \mbox{ is a regular gem of } M\};$$
the Gurau degree (G-degree in short) $\mathcal{D}_G (M)$ of $M$ is defined as
$$\mathcal{D}_G (M)=\min \{\omega_G (\Gamma) \ | \  (\Gamma,\gamma) \mbox{ is a regular gem of } M\};$$
the weighted Gurau degree (weighted G-degree in short) $\tilde{\mathcal{D}}_G (M)$ of $M$ is defined as
$$\tilde{\mathcal{D}}_G (M)=\min \{\omega_G (\Gamma) \ | \  (\Gamma,\gamma) \in \tilde{\mathbb{G}}_d  \mbox{ is a regular gem of } M\};$$

\end{definition}
If $M$ is a closed $d$-manifold then it is clear that $\mathcal G(M)=\tilde{\mathcal G}(M)=\bar{\mathcal G}(M)$ and $\tilde{\mathcal{D}}_G (M)= \mathcal{D}_G (M)$. Also, if $M$ is a $d$-manifold with at most one boundary component then it is easy to see that $\tilde{\mathcal G}(M)=\bar{\mathcal G}(M)$ and $\tilde{\mathcal{D}}_G (M)= \mathcal{D}_G (M)$. In general,  $\tilde{\mathcal G}(M) \geq \bar{\mathcal G}(M)$ and $\tilde{\mathcal{D}}_G (M) \geq \mathcal{D}_G (M)$.
 
 \begin{lemma}\label{lemma:gtilde}
 Let $(\Gamma,\gamma) \in \mathbb{G}_d$ represent a compact connected $d$-manifold $M$ with boundary. Choose an arbitrary fixed color $c \in \Delta_{d-1}$. Let $(\tilde{\Gamma},\tilde{\gamma})$ be the corresponding regular gem for $M$ with $c$ as a singular color. Let $\tilde{g}_{ij}$ and $\partial g_{kl}$ be the number of connected components of $\tilde{\Gamma}_{\{i,j\}}$ and $\partial \Gamma_{\{k,l\}}$ respectively. Then
 $$\tilde{g}_{id}=\dot{g}_{id}+ \partial g_{ic}, \mbox{ } i \neq c \in \Delta_{d-1},$$
 and
 $$\tilde{g}_{cd}=g_{cd}= \dot{g}_{cd}+ \bar{p},$$ where $2\bar{p}$ denotes the number of boundary vertices in $(\Gamma,\gamma).$
 \end{lemma}
 \begin{proof}
 Let  $i \neq c \in \Delta_{d-1}$. From the construction of $(\tilde{\Gamma},\tilde{\gamma})$, it is clear that any $\{i,d\}$-colored cycle in $(\Gamma,\gamma)$ is also a  $\{i,d\}$-colored cycle  in $(\tilde{\Gamma},\tilde{\gamma})$. But, there will be some additional  $\{i,d\}$-colored cycles in $(\tilde{\Gamma},\tilde{\gamma})$, and each of them is obtained by connecting  some $\{i,d\}$-colored walks in $(\Gamma,\gamma)$ by $d$-colored edges. We claim that these additional cycles in $(\tilde{\Gamma},\tilde{\gamma})$ have one-one correspondence with the $\{i,c\}$-colored cycles in $(\partial\Gamma,\partial\gamma)$, i.e.,  $\tilde{g}_{id}-\dot{g}_{id}=\partial g_{ic}$. From the construction of $(\tilde{\Gamma},\tilde{\gamma})$, it is easy to see that any $c$-colored edge between two vertices $u_1$ and $u_2$ in $\partial\Gamma$ corresponds to a $d$-colored edge between the vertices $u_1$ and $u_2$ in $\tilde \Gamma$ and any $i$-colored edge between two vertices $v_1$ and $v_2$ in $\partial\Gamma$ corresponds to a $\{i,d\}$-colored walk between the vertices $v_1$ and $v_2$ in $\tilde \Gamma$. Thus, there is a one-one correspondence between the $\{i,c\}$-colored cycles in $(\partial\Gamma,\partial\gamma)$ and the  $\{i,d\}$-colored cycles in $(\tilde{\Gamma},\tilde{\gamma})$ which is obtained by connecting some $\{i,d\}$-colored walks in $(\Gamma,\gamma)$ by $d$-colored edges. Thus, $\tilde{g}_{id}=\dot{g}_{id}+ \partial g_{ic}$.
 
 From the construction of $(\tilde{\Gamma},\tilde{\gamma})$, it is clear that a $\{c,d\}$-colored cycle  in $(\tilde{\Gamma},\tilde{\gamma})$ is either a $\{c,d\}$-colored cycle in $(\Gamma,\gamma)$ or is obtained from a $\{c,d\}$-colored walk in $(\Gamma,\gamma)$ by adding an edge of color $d$ between the boundary vertices. Therefore, $\tilde{g}_{cd}=g_{cd}$. Further, each $\{c,d\}$-colored walk in $(\Gamma,\gamma)$ corresponds to a $c$-colored edge in $(\partial\Gamma,\partial\gamma)$. Since the number of boundary vertices in  $(\Gamma,\gamma)$  is $2\bar{p}$, the number of $c$-colored edge in $(\partial\Gamma,\partial\gamma)$ is $\bar{p}$. Therefore,
 $\tilde{g}_{cd}=g_{cd}= \dot{g}_{cd}+ \bar{p}$.
  \end{proof}
  
 \begin{corollary}\label{cor:rho}
  Let $(\Gamma,\gamma) \in \mathbb{G}_d$ represent a compact connected $d$-manifold $M$ with boundary. Choose an arbitrary fixed color $c \in \Delta_{d-1}$. Let $(\tilde{\Gamma},\tilde{\gamma})$ be the corresponding regular gem for $M$ with $c$ as a singular color. Then for a permutation  $\varepsilon=(\varepsilon_0,\varepsilon_1,\dots,\varepsilon_d =d)$ of $\Delta_d$, we have
  \begin{enumerate}[$(i)$] 
   \item $\rho_\varepsilon (\tilde{\Gamma})=\rho_\varepsilon (\Gamma)$ if $c \in \{\varepsilon_0,\varepsilon_{d-1} \}$,
\item $\rho_ \varepsilon (\tilde{\Gamma})=\rho_\varepsilon (\Gamma)+\partial g_{\varepsilon_0\varepsilon_{d-1}}-\partial g_{\varepsilon_0 \varepsilon_{d-1} c}$ if  $c \notin \{\varepsilon_0,\varepsilon_{d-1} \}$, where $\partial g_{i_0\cdots i_k}=g(\partial \Gamma_{\{{i_0\cdots i_k}\}})$.
  \end{enumerate} 
 \end{corollary}
 
 \begin{proof} 
\begin{align*}
2\rho_ \varepsilon(\Gamma) &= 2-\sum_{i \in \mathbb{Z}_{d+1}}\dot{g} _{\varepsilon_i\varepsilon_{i+1}} - (1-d) \ \dot{p} - (2-d) \ \overline{p} -\partial g_{\varepsilon_0\varepsilon_{d-1}}\\
&=  2-(\sum_{i \in \mathbb{Z}_{d+1}}\tilde{g} _{\varepsilon_i\varepsilon_{i+1}} - \overline{p} -\partial g_{\varepsilon_0\varepsilon_{d-1}})- (1-d) \ \dot{p} - (2-d) \ \overline{p}  -\partial g_{\varepsilon_0\varepsilon_{d-1}}\\
&= 2-\sum_{i \in \mathbb{Z}_{d+1}}\tilde{g} _{\varepsilon_i\varepsilon_{i+1}}- (1-d) \ \dot{p} - (1-d) \ \overline{p}= 2-\sum_{i \in \mathbb{Z}_{d+1}}\tilde{g} _{\varepsilon_i\varepsilon_{i+1}}- (1-d) \ p=2\rho_ \varepsilon(\tilde{\Gamma}).
\end{align*}
This proves Part $(i)$. Further for any $i,j,k\in\Delta_{d-1}$, $\partial g_{ij}+\partial g_{jk}+\partial g_{ik}=2+\bar p$. Thus,
\begin{align*}
2\rho_ \varepsilon(\Gamma) &= 2-\sum_{i \in \mathbb{Z}_{d+1}}\dot{g} _{\varepsilon_i\varepsilon_{i+1}} - (1-d) \ \dot{p}  - (2-d) \ \overline{p}  -\partial g_{\varepsilon_0\varepsilon_{d-1}}\\
 &=  2-(\sum_{i \in \mathbb{Z}_{d+1}}\tilde{g} _{\varepsilon_i\varepsilon_{i+1}} -\partial g_{\varepsilon_0 c}- \partial g_{c \varepsilon_{d-1}})- (1-d) \ p - \overline{p} -\partial g_{\varepsilon_0\varepsilon_{d-1}}\\
&= 2-\sum_{i \in \mathbb{Z}_{d+1}}\tilde{g} _{\varepsilon_i\varepsilon_{i+1}} +(\partial g_{\varepsilon_0 c}+ \partial g_{c \varepsilon_{d-1}}+ \partial g_{\varepsilon_0 \varepsilon_{d-1}})- (1-d) \ p - \overline{p} -2\partial g_{\varepsilon_0\varepsilon_{d-1}}\\
 &=2-\sum_{i \in \mathbb{Z}_{d+1}}\tilde{g} _{\varepsilon_i\varepsilon_{i+1}} +2 \partial g_{\varepsilon_0 \varepsilon_{d-1} c}- (1-d) \ p -2\partial g_{\varepsilon_0\varepsilon_{d-1}}\\
 &= 2\rho_ \varepsilon(\tilde{\Gamma})+2 \partial g_{\varepsilon_0 \varepsilon_{d-1} c} -2\partial g_{\varepsilon_0\varepsilon_{d-1}}.\\
 \end{align*}
This proves Part $(ii)$.
\end{proof}

 \begin{theorem}\label{thm:genrelweighted}
 Let $M$ be a compact connected $d$-manifold with boundary. Then, $\mathcal G(M) \geq \tilde{\mathcal G}(M) \geq \bar{\mathcal G}(M)$ and $\tilde{\mathcal{D}_G} (M)\geq \mathcal{D}_G (M)$.
 \end{theorem}
 
 \begin{proof}
 From Definition \ref{maindef}, it is clear that $\tilde{\mathcal G}(M) \geq \bar{\mathcal G}(M)$ and $\tilde{\mathcal{D}_G} (M)\geq \mathcal{D}_G (M)$. Let $(\Gamma,\gamma) \in \mathbb{G}_d$ be a non-regular graph such that the regular genus of $M$ is assumed by the regular genus of $\Gamma$. Let the regular genus $\rho(\Gamma)$ of $\Gamma$ be attained with respect to the cyclic permutation $\varepsilon=(\varepsilon_0,\dots,\varepsilon_d=d)$ of $\Delta_d$. Let $(\tilde{\Gamma},\tilde{\gamma})$ be the corresponding regular gem for $M$ with $c \in \{ \varepsilon_0,\varepsilon_{d-1}\}$ as a singular color. Then by Corollary \ref{cor:rho}, $\rho_\varepsilon (\tilde{\Gamma})=\rho_\varepsilon (\Gamma)=\mathcal G(M)$. Let $\varepsilon'$ be the permutation by interchanging the colors $c$ and $d$ in $\varepsilon$ and $(\tilde{\Gamma}',\tilde{\gamma}')$ be the graph by interchanging the colors $c$ and $d$ in $(\tilde{\Gamma},\tilde{\gamma})$. Then $(\tilde{\Gamma}',\tilde{\gamma}')$ is a regular gem of $M$ in $\tilde{\mathbb{G}}_d$,  and $\rho_{\varepsilon'} (\tilde{\Gamma}')= \mathcal G(M)$. Therefore,  $\mathcal G(M) \geq \tilde{\mathcal G}(M)$.
  \end{proof}
 
\begin{proposition} [\cite{cc19, ch86}] \label{prop:rank}
Let $(\Gamma,\gamma)$ be a $(d+1)$-colored graph representing the singular $d$-manifold M and the associated compact $d$-manifold $\check{M}$. Let $X_{ij}$ and $R_{ij}$ be the sets which are in bijective correspondence to the connected components of $\Gamma_{\Delta_d \setminus \{i,j\}}$ and $\Gamma_{ij}$ respectively. Let $\bar{R}_{ij}\subset X_{ij}$ corresponds to a maximal tree of the subcomplex $K_{ij}$ of $\mathcal{K}(\Gamma)$ (consisting only of vertices labeled by $i$ and $j$, and edges connecting them). Then,
\begin{enumerate}[$(a)$] 
\item if $i,j\in \Delta_d$ are not singular in $\Gamma$,
$$\pi_1 (\check{M})=<X_{ij}/ R_{ij} \cup \bar{R}_{ij}>,$$
\item if no color in $\Delta_d \setminus \{ i,j\}$ is singular in $\Gamma$. Then
$$\pi_1 (M)=<X_{ij}/ R_{ij} \cup \bar{R}_{ij}>.$$
\end{enumerate}
\end{proposition}
 
 From now onwards, we shall focus on connected compact $4$-manifolds with non-spherical boundary components.

\begin{remark}\label{rem:chi}
{\rm
Let $(\Gamma,\gamma)$ be a regular $5$-colored graph. Then the Euler characteristic of the corresponding simplicial cell complex does not change by canceling $1$-dipole in $\Gamma$.
Because with the cancellation of one $1$-dipole in $\Gamma$, the number of $f$-vectors in the corresponding simplicial cell complex are reduced by
$$f_0=1, f_1=4,f_2=6,f_3=5,f_4=2,$$
and hence the result follows. }
\end{remark}

\begin{remark}\label{rem:plush}
{\rm
Let $M$ be a compact $4$-manifold with boundary such that $\partial^1 M,\dots,\partial^h M$ are components of $\partial M$. Let $\widehat{M}$ be the singular $4$-manifold obtained by coning off each component of $\partial M$. Let $\mathbb{H}_ {\partial^i M}$ denote the cone over $\partial^i M$ for $i \in \{ 1,2,\dots,h \}$. Then $\chi(\mathbb{H}_ {\partial^i M} )=1$.
Since $\partial M $ is a $3$-manifold, $\chi(\partial M)=0$. Therefore, 
 $\chi(\widehat{M})=\chi(M)+ \sum_ {i=1} ^ h \chi(\mathbb{H}_ {\partial^i M} )- \chi(\partial M)$ = $\chi(M)+h$.}
 \end{remark}
 
 \begin{remark}\label{rem:genus}
 {\rm
Let $(\Gamma,\gamma)$ be a regular $5$-colored graph then the regular genus of the graph does not change by canceling $1$-dipoles in $\Gamma$.}
\end{remark}

\begin{remark}\label{rem:fg}
{\rm
Let $(\Gamma,\gamma)$ be a $5$-colored graph (possibly with boundary) representing the simplicial cell-complex $\mathcal{K}(\Gamma)$. Let $(\Gamma',\gamma')$ be the $5$-colored graph after canceling of $1$-dipole from  $(\Gamma,\gamma)$. Let $\mathcal{K}(\Gamma')$ be the simplicial cell-complex corresponding to $(\Gamma',\gamma')$. Then $|\mathcal{K}(\Gamma')|$ is a deformation retract of $|\mathcal{K}(\Gamma)|$. Thus, $|(\Gamma',\gamma')|$ and $|\mathcal{K}(\Gamma')|$ have the same fundamental group.}
 \end{remark}
 
 Let $(\hat \Gamma,\hat \gamma) \in \tilde{\mathbb{G}}_4$ be a regular gem of a compact connected $4$-manifold $M$ with empty or connected boundary. Let  $m$ and $\hat{m}$ be the rank of the fundamental groups of $M$ and the corresponding singular manifold $\widehat{M}$ respectively. Let $\hat g_{ijk}=g(\hat \Gamma_{\{i,j,k\}})$. It follows  from Proposition \ref{prop:rank} that  $\hat g_{ijk} \geq \hat{m}+1$ and $\hat g_{ij4} \geq m+1$ for $\{ i,j,k \} \in \Delta_3$. Let $\hat g_{ijk} = (\hat{m}+1)+\hat t_{ijk}$ and $\hat g_{ij4} = (m+1)+\hat t_{ij4}$ where $\hat t_{ijk}\geq 0$.
 
 \begin{proposition} [\cite{ccg20}] \label{prop:rhoepsilon}
 Let $(\hat \Gamma,\hat \gamma) \in \tilde{\mathbb{G}}_4$ be a regular gem of a compact connected $4$-manifold $M$ with empty or connected boundary. Let $\widehat{M}$, $m$, $\hat{m}$, $\hat t_{ijk}$ be as above. Then for any cyclic permutation $\varepsilon=(\varepsilon_0,\dots,\varepsilon_4)$,
$$\rho_\varepsilon (\hat \Gamma)=2\chi(\widehat{M})+5m-2(m-\hat{m})-4+ \sum_{i \in \mathbb{Z}_5} \hat t_{\varepsilon_i \varepsilon_{i+2} \varepsilon_{i+4}}.$$
 \end{proposition}
 
 \begin{proposition}[\cite{ccg20}] \label{prop:omegasum}
 Let $(\hat \Gamma,\hat \gamma)$ be a regular $5$-colored graph in $ \tilde{\mathbb{G}}_4$. For any cyclic permutation  $\varepsilon=(\varepsilon_0,\dots,\varepsilon_4=4)$ of $\Delta_4$, let $\varepsilon' =(\varepsilon_1,\varepsilon_3,\varepsilon_0,\varepsilon_2,\varepsilon_4=4)$. Then, 
 $$\omega_G (\hat \Gamma)=6(\rho_\varepsilon(\hat \Gamma)+\rho_{\varepsilon'} (\hat \Gamma)).$$ 
 \end{proposition}
 
 \begin{theorem}\label{thm:weightedgenus}
 Let $M$ be a compact connected PL $4$-manifold with $h$ boundary components. Let $\widehat{M}$ be the singular $4$-manifold obtained from $M$ by coning off each boundary component of $M$. Let $m$ and $\hat{m}$ be the rank of fundamental groups of $M$ and $\widehat{M}$ respectively. Then,
 $$\tilde{G} (M) \geq 2 \chi(M)+3m+2h-4+2 \hat{m}$$ and
 $$\tilde{D}_G (M) \geq 12(2 \chi(M)+3m+2h-4+2 \hat{m}).$$
 \end{theorem}
 
 \begin{proof}
Let $\partial^1 M, \partial ^2 M,\dots, \partial^h M$ be the boundary components of $M$. Let $(\hat{\Gamma},\hat{\gamma})\in \tilde{G}_4$ be any regular gem of $M$, i.e., it represents the singular manifold $\widehat M$.
 Let $(\hat{\Gamma}',\hat{\gamma}')\in \tilde{G}_4$ be the regular $5$-colored graph obtained from $(\hat{\Gamma},\hat{\gamma})$  by canceling all possible $1$-dipoles of each color labeled by $i \in \Delta_{4}$.
Then, $\mathcal{K}(\hat{\Gamma}')$ has $5$ vertices such that the vertices labeled by color $i \in \Delta_{3}$ are non singular and the vertex labeled by color $4$ is singular. Further, the link of the vertex labeled by color $4$ in $|\mathcal{K}(\hat{\Gamma}')|$ is $\partial^1 M \#\dots\# \partial^h M$. Let $\widehat{M}^ \prime=|\mathcal{K}(\hat{\Gamma}')|$. Then from Remarks \ref{rem:chi} and \ref{rem:plush}, we have $\chi(\widehat{M}')=\chi(\widehat{M})=\chi(M)+h$. Let $M'$ be the manifold with connected boundary obtained by removing a small neighborhood of  vertex $4$ from $\widehat{M}'$. Then Remark \ref{rem:fg} implies that rank($\pi_1 (\widehat{M}^ \prime)$) = rank($\pi_1 (\widehat{M})$) = $\hat{m}$. Further, it is easy to see that $M$ is a deformation retract of $M'$. Thus,  rank($\pi_1 (M')$) = rank($\pi_1 (M)$) = $m$. For $\{i_0,\dots,i_k\} \subset \Delta_4$, let $\hat g_{i_0\cdots i_k}'$ denote the number of connected components of $\hat{\Gamma}_{\{i_0,\dots,i_k\}}'$. It follows from Proposition \ref{prop:rank} that $\hat g_{ijk}' \geq \hat{m}+1$ and $\hat g_{ij4}' \geq m+1$ for $\{ i,j,k \} \in \Delta_3$. We consider $\hat g_{ijk}'= (\hat{m}+1)+\hat t_{ijk}'$ and $\hat g_{ij4}' = (m+1)+\hat t_{ij4}'$ where $\hat t_{ijk}' \geq 0$.

 For $\varepsilon=(\varepsilon_0,\dots,\varepsilon_4)$, let $\varepsilon'=(\varepsilon_1,\varepsilon_3,\varepsilon_0,\varepsilon_2,\varepsilon_4=4)$. It follows from Proposition \ref{prop:rhoepsilon} that
  \begin{equation}\label{eq:rho}
  \rho_ \varepsilon(\hat{\Gamma})= \rho_ \varepsilon(\hat{\Gamma}')= 2 \chi(\widehat{M}')+5m-2(m-\hat{m})-4+\sum_{i\in \mathbb{Z}_5}  \hat t_{\varepsilon_i \varepsilon_{i+2} \varepsilon_{i+4}}'
  \end{equation}
  and $$\rho_ {\varepsilon'} (\hat{\Gamma})= \rho_ {\varepsilon'} (\hat{\Gamma}')= 2 \chi(\widehat{M}')+5m-2(m-\hat{m})-4+\sum_{i\in \mathbb{Z}_5} \hat t_{\varepsilon_i \varepsilon_{i+1} \varepsilon_{i+2}}'.$$
  Then by Proposition \ref{prop:omegasum} we have
  \begin{align*}
   \omega_G (\hat{\Gamma})&=6(\rho_ \varepsilon(\hat{\Gamma})+\rho_ {\varepsilon'} (\hat{\Gamma}))\\
   &=6(4 \chi(\widehat{M}')+10m-4(m-\hat{m})-8+\sum_{0\leq i<j<k\leq 4} \hat t_{ijk}')\\
   &=6(4 \chi(M)+4h+10m-4(m-\hat{m})-8+\sum_{0\leq i<j<k\leq 4} \hat t_{ijk}').
\end{align*}
Thus, for any $(\hat{\Gamma},\hat{\gamma}) \in \tilde{\mathbb{G}}_d$, $$ \omega_G (\hat{\Gamma})\geq 12(2 \chi(M)+3m+2h-4+2 \hat{m}).$$
 Since 
$$\tilde{\mathcal{D}}_G (M)=\min \{\omega_G (\hat{\Gamma}) \ | \  (\hat{\Gamma},\hat{\gamma}) \in \tilde{\mathbb{G}}_d  \mbox{ is a regular gem of } M\},$$
we have $$\tilde{D}_G (M) \geq 12(2 \chi(M)+3m+2h-4+2 \hat{m}).$$

Further, it follows from Eq. \eqref{eq:rho} that, for any regular gem $(\hat{\Gamma},\hat{\gamma}) \in \tilde{\mathbb{G}}_d$ of $M$,
 $$\rho(\hat{\Gamma}) \geq \rho_ \varepsilon (\hat{\Gamma})\geq 2 \chi(M)+3m+2h-4+2 \hat{m}.$$
Since $$\tilde{\mathcal G}(M) = \min \{\rho(\hat{\Gamma}) \ | \  (\hat{\Gamma},\hat{\gamma}) \in \tilde{\mathbb{G}}_d \mbox{ is a regular gem of } M\},$$
we have $$\tilde{\mathcal{G}}(M) \geq 2 \chi(M)+3m+2h-4+2 \hat{m}.$$
 \end{proof}
 
 The above result in combination with Theorem \ref{thm:genrelweighted} gives a lower bound for the regular genus of compact connected $4$-manifolds $M$ with boundary, which  is stronger than the previous known lower bounds for the regular genus of $M$.
\begin{corollary}
Let $M$ be a connected compact $4$-manifold with $h$ boundary components. Then, the regular genus of $M$ satisfies the following inequality: $$\mathcal{G}(M) \geq 2 \chi(M)+3m+2h-4+2 \hat{m},$$
where $m$ and $\hat{m}$ are the rank of fundamental groups of $M$ and the corresponding singular manifold $\widehat{M}$ respectively.
\end{corollary} 

 \begin{definition}
 Let $M$ be a connected compact  $4$-manifold with $h$ boundary components and $\widehat{M}$ be the corresponding singular $4$-manifold. Let $(\hat{\Gamma},\hat{\gamma}) \in \tilde{G}_4$ be the regular gem of $M$ such that $g(\hat \Gamma_ {\hat{d}})=h$ and $g(\hat \Gamma_ {\hat{c}})=1$ for each $c \in \Delta_3$. For $\{i_0,\dots,i_k\}\subset \Delta_4$, let $g_{i_0\cdots i_k}$ denote the number of connected components of $\hat{\Gamma}_{\{i_0,\dots,i_k\}}$. Let $m$ and $\hat{m}$ be the rank of fundamental groups of $M$ and $\widehat{M}$ respectively. Then, $\hat{\Gamma}$ is said to be semi-simple if 
 $$g_{ijk}=\hat{m}+h  \mbox{  and  } g_{ij4}=m+1, \mbox{ where } i,j,k \in \Delta_3$$
 and is said to be the weak semi-simple if, for a cyclic permutation $\varepsilon=(\varepsilon_0,\dots,\varepsilon_4)$ of $\Delta_4$,
$$g_{\varepsilon_i \varepsilon_{i+2} \varepsilon_{i+4}}=\hat{m}+h, \mbox{ where }  i \in \{ 1,3 \} \mbox{  and  } g_{\varepsilon_i \varepsilon_{i+2} \varepsilon_{i+4}}=m+1, \mbox{ where }  i \in \{ 0,2,4 \}.$$ 

 \end{definition}
 
 \begin{definition}
 Let $M$ be a connected compact  $4$-manifold with boundary. Then $M$ is said to admit a semi-simple gem if there exists a regular gem $(\tilde{\Gamma},\tilde{\gamma}) \in \tilde{G}_4$ of $M$, which is semi-simple, and $M$ is said to admit a weak semi-simple gem if there exists a regular gem $(\tilde{\Gamma},\tilde{\gamma}) \in \tilde{G}_4$ of $M$, which is weak semi-simple.
 \end{definition}
 
 Let $M$ be a connected compact  $4$-manifold with $h$ boundary components. Then we always have a regular gem $(\hat{\Gamma},\hat{\gamma}) \in \tilde{G}_4$  of $M$ such that $g(\hat \Gamma_ {\hat{d}})=h$ and $g(\hat \Gamma_ {\hat{i}})=1$ for each $i \in \Delta_3$. Let $(\Gamma,\gamma)$ be a crystallization of $M$. Let $(\tilde{\Gamma},\tilde{\gamma}) \in \tilde{G}_4$ be the corresponding regular gem of $M$. Then by removing all $1$-dipoles of color $i \in \Delta_3$, let $(\hat{\Gamma},\hat{\gamma})$ be the resulting graph. Then $(\hat{\Gamma},\hat{\gamma}) \in \tilde{G}_4$, $g(\hat \Gamma_ {\hat{d}})=h$ and $g(\hat \Gamma_ {\hat{i}})=1$ for each $i \in \Delta_3$.

 \begin{theorem}\label{theorem:sem-simple}
 Let $M$ be a connected compact $4$-manifold with $h$ boundary components. Let $m$ and $\hat{m}$ be the rank of the fundamental groups of $M$ and the corresponding singular manifold $\widehat{M}$ respectively. Then, $M$ admits a semi-simple gem if and only if $\tilde{D}_G (M) = 12(2 \chi(M)+3m+2h-4+2 \hat{m})$ and $M$ admits a weak semi-simple gem if and only if  $\tilde{G}(M) = 2 \chi(M)+3m+2h-4+2 \hat{m}$.
 \end{theorem}
 
 \begin{proof}
  Let $(\hat{\Gamma},\hat{\gamma}) \in \tilde{G}_4$ be a regular gem of $M$ such that $g(\hat \Gamma_ {\hat{d}})=h$ and $g(\hat \Gamma_ {\hat{i}})=1$ for each $i \in \Delta_3$. Let  $(\hat{\Gamma}',\hat{\gamma}') \in \tilde{G}_4$ be the graph obtained from $(\hat{\Gamma},\hat{\gamma})$   by canceling all the $(h-1)$ number of $1$-dipoles of color $4$, and let $\widehat{M}'=|\mathcal{K}(\hat{\Gamma}')|$.  Let $M'$ be the manifold with connected boundary obtained by removing a small neighborhood of  vertex $4$ from $\widehat{M}'$. It follows from Remark \ref{rem:fg} that rank($\pi_1 (\widehat{M}^ \prime)$) = rank($\pi_1 (\widehat{M})$) = $\hat{m}$. Since $M$ is a deformation retract of $M'$, rank($\pi_1 (M')$) = rank($\pi_1 (M)$) = $m$. For $\{i_0,\dots,i_k\} \subset \Delta_4$, let $\hat g_{i_0\cdots i_k}$ and $\hat g_{i_0 \cdots i_k}'$ denote the number of components of $\hat{\Gamma}_{\{i_0,\dots,i_k\}}$ and $\hat{\Gamma}_{\{i_0,\dots,i_k\}}'$ respectively.
Further, Proposition \ref{prop:rank}  implies,  $\hat g_{ijk}-h+1= \hat g_{ijk}'\geq \hat{m}+1$ and $\hat g_{ij4}= \hat g_{ij4}' \geq m+1$, for $i,j,k \in \Delta_3$. We write, $\hat g_{ijk}-h+1= \hat g_{ijk}' = (\hat{m}+1) + \hat t_{ijk}'$ and $\hat g_{ij4} =\hat g_{ij4}'= (m+1)+\hat t_{ij4}'$. Therefore,

\begin{align*}
 M \mbox{ admits a semi-simple gem} \Leftrightarrow & ~  \mbox{there exists a regular gem } (\hat{\Gamma},\hat{\gamma}) \in \tilde{G}_4 \mbox{ of } M \mbox{ such that } \\
 &~ \sum_{0\leq i<j<k\leq 4} \hat t_{ijk}'=0\\
\Leftrightarrow  & ~ \omega_G (\hat{\Gamma}) = 12(2 \chi(M)+3m+2h-4+2 \hat{m})\\
\Leftrightarrow & ~ \tilde{D}_G (M) = 12(2 \chi(M)+3m+2h-4+2 \hat{m}).
\end{align*}

On the other hand,
\begin{align*}
 M \mbox{ admits a weak semi-simple gem }  \Leftrightarrow & ~  \mbox{there exists a regular gem } (\hat{\Gamma},\hat{\gamma}) \in \tilde{G}_4 \mbox{ of } M \mbox{ such} \\
 & ~ \mbox{that } \sum_{i \in \mathbb{Z}_5} \hat t_{\varepsilon_i \varepsilon_{i+2} \varepsilon_{i+4}}' =0 \mbox{ for some } \varepsilon=(\varepsilon_0,\dots,\varepsilon_4)\\
 \Leftrightarrow & ~   \rho_\varepsilon (\hat{\Gamma}) = 2 \chi(M)+3m+2h-4+2 \hat{m}\\
 \Leftrightarrow & ~  \tilde{\mathcal{G}} (M) = 2 \chi(M)+3m+2h-4+2 \hat{m}.
\end{align*}
 \end{proof}
 
For a connected compact $4$-manifold $M$ with boundary, let us define its {\em weighted gem-complexity} as the non-negative integer $\tilde{\mathit k}(M) =
p - 1$, where $2p$ is the minimum order of a regular gem of $M$.
 
 \begin{corollary}
 Let $M$ be a connected compact $4$-manifold with $h$ boundary components. Let $m$ be the rank of the fundamental group of $M$. Then
$$\tilde{D}_G (M)=6(\chi (M)-1+\tilde{\mathit{k}} (M)).$$
\end{corollary}

\begin{proof}
  Let $(\hat{\Gamma},\hat{\gamma}) \in \tilde{G}_4$ be a regular of $M$ such that $g(\hat \Gamma_ {\hat{d}})=h$ and $g(\hat \Gamma_ {\hat{i}})=1$ for each $i \in \Delta_3$. Let $2 p$ be the number of vertices of $\hat{\Gamma}$. Let $\hat{\Gamma}^\prime$, $\widehat{M}^\prime$, $m,\hat m$, $\hat g_{ijk}',\hat t_{ijk}'$ be as in Theorem \ref{theorem:sem-simple}. Let $2 p^\prime$ be the number of vertices of $\hat{\Gamma}^\prime$. Then $2p=2p^\prime +2(h-1)$. Then, from the proof of Theorem \ref{thm:weightedgenus}, we have
$$ \omega_G (\hat{\Gamma})=6(4 \chi(M)+4h+6m+4\hat{m}-8+\sum_{0\leq i<j<k\leq 4} \hat t_{ijk}').$$
Further by the Dehn–Sommerville equations in dimension four (cf. \cite[Lemma 6]{bb18}) we have 
$$2p'=6\chi(\widehat{M}')+2\sum_{0\leq i<j<k\leq 4} \hat g_{ijk}'-30=6(\chi(M)+h)+12m+8\hat m-10 +2\sum_{0\leq i<j<k\leq 4} \hat t_{ijk}'.$$
Therefore,
\begin{align*}
\omega_G({\hat{\Gamma}})&=6(4 \chi(M)+4h+6m+4\hat{m}-8+p-h+1-3\chi(M)-3h-6m-4\hat m+5)\\
&=6(\chi(M)-1+p-1).
\end{align*}
This proves the result.
\end{proof}
 
 \bigskip
 
 \bigskip

\noindent {\bf Acknowledgement:} The author would like to thank the anonymous referees for many useful comments and suggestions.
The first author is supported by DST INSPIRE Faculty Research Grant, India (Award No. DST/INSPIRE/04/2017/002471).
%
%

{\footnotesize

\end{document}